\documentclass[10 pt]{article}
\usepackage{amssymb,amsthm}
\font\smallrm=cmr8
\font\bigmath=cmsy10 scaled \magstep 2
\newcommand{\emp}{\emptyset}
\newcommand{\ben}{\mathbb N}
\newcommand{\ber}{\mathbb R}
\newcommand{\beq}{\mathbb Q}
\newcommand{\bez}{\mathbb Z}

\newcommand{\nhat}[1]{\{1,2,\ldots,#1\}}
\newcommand{\ohat}[1]{\{0,1,\ldots,#1\}}
\newcommand{\rank}{\hbox{\rm rank}}
\newcommand{\sgn}{\hbox{\rm sgn}}
\newcommand{\bigtimes}{\hbox{\bigmath \char'2}} 
\newtheorem{theorem}{Theorem}[section]
\newtheorem{corollary}[theorem]{Corollary}
\newtheorem{lemma}[theorem]{Lemma}
\newtheorem{question}[theorem]{Question}
\theoremstyle{definition}
\newtheorem{definition}[theorem]{Definition}

\title{Duality for image and kernel partition regularity of infinite matrices}
\date{}
\author{Neil Hindman
        \footnote{Department of Mathematics,
                 Howard University,
                  Washington, DC 20059, USA.\hfill\break
                  {\tt nhindman@aol.com}}
        \thanks{This author acknowledges support received from the National
                Science Foundation (USA) via Grant DMS-1460023.}
\and
Imre Leader
        \footnote{Department of Pure Mathematics and Mathematical Statistics,
Centre for Mathematical Sciences,
Wilberforce Road, Cambridge, CB3K.\hfill\break
                  {\tt leader@dpmms.cam.ac.uk}}
\and
Dona Strauss
        \footnote{Department of Pure Mathematics,
        University of Leeds,
        Leeds LS2 9J2, UK. \hfill\break
        {\tt d.strauss@hull.ac.uk}}
}

\begin{document}

\maketitle

\begin{abstract}  A matrix $A$ is 
{\it image partition regular over $\beq$\/} provided that
whenever $\beq\setminus\{0\}$ is finitely coloured, there is a vector $\vec x$
with entries in $\beq\setminus\{0\}$ such that the entries of $A\vec x$ are monochromatic.  It is 
{\it kernel partition regular over $\beq$\/} provided that
whenever $\beq\setminus\{0\}$ is finitely coloured, the matrix has a monochromatic member of its kernel.
We establish a duality for these notions valid for both finite and infinite matrices.
We also investigate the extent to which this duality holds for 
matrices partition regular over proper subsemigroups of $\beq$.
\end{abstract}

\section{Introduction}

We let $\ben$ be the set of positive integers and
$\omega=\ben\cup\{0\}$.
We shall be concerned throughout this paper with
matrices (finite or infinite) that have rational
entries and finitely many nonzero entries per row. We shall
also assume that every matrix that we consider has this property.
(However, elements of $\beq^{\omega}$ are not assumed to 
have finitely many nonzero entries.)

For consistency of treatment between
the finite and infinite cases, we shall treat $u\in\ben$ as an ordinal.
Consequently $u=\ohat{u-1}$ and $\omega=\{0,1,2,\ldots\}$ is the first infinite ordinal.
Thus, if $u,v\in\ben\cup \{\omega \}$, and $A$ is a $u\times v$ matrix,
the rows and columns of $A$ will be indexed by $u=\{i:i<u\}$ and
$v=\{i:i<v\}$, respectively. (See \cite{J} for an
introduction to ordinals.)

As is standard in Ramsey Theory, a {\it finite colouring\/} of a 
set $X$ is a function whose domain is $X$ and whose range is finite.
Given a colouring $f$ of $X$, a subset $B$ of $X$ is {\it monochromatic\/}
if and only if $f$ is constant on $B$.

\begin{definition}\label{defiprkpr} Let $u,v\in\ben\cup\{\omega\}$,
let $A$ be a $u\times v$ matrix,  let $S$ be a nontrivial subsemigroup
of $(\beq,+)$, and let $G$ be the subgroup of $\beq$ generated by $S$.
\begin{itemize}
\item[(a)] The matrix $A$ is {\it kernel partition regular over $S$\/}
if and only if whenever $S\setminus\{0\}$ is finitely coloured,
there exists $\vec x\in (S\setminus\{0\})^v$ such that
$A\vec x=\vec 0$ and the entries of $\vec x$ are monochromatic.
\item[(b)] The matrix $A$ is {\it image partition regular over $S$\/}
if and only if whenever $S\setminus\{0\}$ is finitely coloured,
there exists $\vec x\in (S\setminus\{0\})^v$ such that
the entries of $A\vec x$ are monochromatic.
\item[(c)] The matrix $A$ is {\it weakly image partition regular over $S$\/}
if and only if whenever $S\setminus\{0\}$ is finitely coloured,
there exists $\vec x\in G^v$ such that
the entries of $A\vec x$ are monochromatic.
\end{itemize}\end{definition}

Notice that, since it is $S\setminus\{0\}$ that is being coloured, the
monochromatic entry which is guaranteed has no zero entries.

Many of the classical theorems of Ramsey Theory assert that certain
matrices are kernel partition regular.  For example, 
Schur's Theorem is the assertion that the matrix
$\left(\begin{array}{ccc}1&1&-1\end{array}\right)$ is 
kernel partition regular over $\ben$.  It is also the assertion that
the matrix $\left(\begin{array}{cc}1&0\\
0&1\\ 1&1\end{array}\right)$ is image partition regular over $\ben$.

For another example, the
instance of van der Waerden's Theorem that says that whenever
$\ben$ is finitely coloured, there is a monochromatic length $4$ 
arithmetic progression, is the assertion that 
the matrix
$$\left(\begin{array}{cc}
1&0\\
1&1\\
1&2\\
1&3\end{array}\right)$$
is image partition regular over $\ben$.
The strengthened version that asks that the increment also be the
same colour is the assertion that the matrix
$$\left(\begin{array}{ccccc}
1&-2&1&0&0\\
0&1&-2&1&0\\
1&-1&0&0&1\end{array}\right)$$
is kernel partition regular over $\ben$.  (There is no kernel partition regular
matrix whose kernels consist exactly of length $4$ arithmetic
progressions.  See \cite[Theorem 2.6]{H}.  Also see \cite{H} for
more background on partition regularity.)

In 1933 R. Rado \cite{R} characterised those finite matrices that 
are kernel partition regular over $\ben$.  Rado proved that
a finite matrix with entries from $\beq$ is kernel partition regular
over $\ben$ if and only if it satisfies the {\it columns 
condition\/} and that the same requirement was necessary and sufficient
for kernel partition regularity over $\beq$ and also    
necessary and sufficient
for kernel partition regularity over $\bez$.

\begin{definition}\label{defcc} Let $u,v\in\ben$ and let $A$ be a $u\times v$ matrix with entries 
from $\beq$.  Denote the columns of $A$ by $\langle\vec c_i\rangle_{i=0}^{v-1}$.
  The matrix $A$ satisfies the {\it columns condition\/} if and only if
there exist $m\in\nhat{v}$ and a partition $\langle I_t\rangle_{t=0}^{m-1}$ of 
$\ohat{v-1}$ such that
\begin{itemize}
  \item[(1)] $\sum_{i\in I_0}\vec c_i \ = \ \vec 0$ and
  \item[(2)] for each $t\in \nhat{m-1}$, if any, $\sum_{i\in I_t}\vec c_i$ is a 
linear combination with coefficients from $\beq$ of $\{\vec c_i:i\in\bigcup_{j=0}^{t-1}I_j\}$.
\end{itemize}
\end{definition}

 Call a set $X\subseteq\ben$
{\it large\/} if whenever $A$ is kernel partition regular over $\ben$,
there must exist $\vec x$ with entries in $X$ such that
$A\vec x=\vec 0$.  Rado conjectured that if a large set is finitely
coloured, there must exist a monochromatic large subset.
Rado's conjecture was proved in 1973 by W. Deuber \cite{D}.  For the
proof, Deuber used a certain collection of image partition regular
matrices, whose images he called {\it $(m,p,c)$-sets\/}.

\begin{definition}\label{defmpc} Let $m,p,c\in \ben$. A matrix
$A$ is an {\it $(m,p,c)$ matrix\/} if and only if $A$ has
$m$ columns and all rows of the form
$\vec r=\left(\begin{array}{cccc}
r_0&r_1&\ldots&r_{m-1}\end{array}\right)$ such that
each $r_i\in\{-p,-p+1,\ldots,p-1,p\}$ and
the first nonzero entry $r_i=c$.\end{definition}

Deuber's proof used three key facts.  First \cite[Satz 2.1]{D},
if $B$ is any finite kernel partition regular matrix, there
exist $m,p,c\in\ben$ such that, if $A$ is an $(m,p,c)$ matrix,
$\vec x\in\ben^m$, and the entries of $A\vec x$ are nonzero, then the set of entries
of $A\vec x$ contain the entries of a vector $\vec y$ such that
$B\vec y=\vec 0$.  Second \cite[Satz 2.2]{D},
if $m,p,c\in\ben$ and $A$ is an $(m,p,c)$ matrix, then there
exists a kernel partition regular matrix $B$ such that, if $\vec y$ has nonzero
entries and $B\vec y=\vec 0$, then the set of entries of $\vec y$ contains
the entries of $A\vec x$ for some $\vec x\in(\bez\setminus\{0\})^m$.
And finally \cite[Satz 3.1]{D}, if a set $X\subseteq\ben$ contains an image
of every $(m,p,c)$-matrix and $X$ is finitely coloured, then there 
is a monochromatic $Y$ that contains an image
of every $(m,p,c)$-matrix.  (When we write ``contains an image'' we really
mean ``contains every entry of an image''.)

In this paper we present  versions of Deuber's duality
between kernel partition regular and image partition regular matrices.
In these results we show in Theorem \ref{BgetsC} that, given a matrix $B$ that
is kernel partition regular over a nontrivial subsemigroup $S$ of $\beq$,  there is a matrix $C$ that is image 
partition regular over $S$ such that
the set of relevant images of $C$ is exactly equal to the set of relevant kernel elements
of $B$.  And, we show in Theorem \ref{AgetsB} that, given a matrix $A$ that is weakly image partition regular over $\beq$, there is
a matrix $B$ that is kernel partition regular over $\beq$ such that the set of relevant kernel elements of
$B$ is exactly equal to the set of relevant images of $A$. (Here ``relevant" means those vectors that could be 
involved in kernel or image partition regularity.)  We also show that the latter result does
not work for proper subgroups of $\beq$.

There are several questions that we are unable to resolve. These are listed at various points.

It is of particular interest that most of our results make no distinction between finite and infinite
matrices.  This is highly unusual.  For example many characterisations of 
finite image partition regular matrices are known, and Rado's characterisation of
finite kernel partition regular matrices is computable.  By way of contrast,
nothing resembling characterisations of infinite kernel or image partition 
regular matrices is known.  See the survey \cite{H} for information on
this point.

We mention in passing that very few infinite partition regular systems are known.
One example is the Finite Sums Theorem \cite{Hb}, which states that whenever the positive
integers are finitely coloured, there exists a sequence $\langle x_n\rangle_{n=0}^\infty$
such that $\{\sum_{n\in F}\,x_n:F$ is a finite nonempty subset of $\omega\}$ is
monochromatic. For other examples see \cite{HLS} and \cite{BHL}.

Also, Rado's conjecture is not valid for 
infinite matrices; there exist infinite image partition regular
matrices $A$ and $B$ and a colouring of $\beq\setminus\{0\}$ in two
colours so that neither colour class contains an image of both $A$ and $B$,
\cite{HSb}.
(And, as noted in \cite{HLS}, the corresponding
statement applies to infinite kernel partition regular matrices. This fact
also follows from Theorem \ref{AgetsB} below.)

We note that {\it duality\/} is not quite the correct word for our results, as there are 
three notions involved, namely kernel partition regularity, image partition regularity, and
weak image partition regularity. However, in the course of the preparation of this paper we 
noticed the following simple fact, which says that the relationship between
image partition regularity and weak image partition regularity is much stronger
than we had realised.  Notice that if $v=\omega$, then the matrix $C$ in the
following theorem is a $u\times\omega$ matrix.  We follow the convention of 
denoting the entries of a matrix by the lower case letter whose upper case
letter is the name of the matrix. 

The content of the following two theorems is probably best understood by
supposing that $C$ is $\left(\begin{array}{cc}A&-A\end{array}\right)$.  We
write it the way we do because, if $v=\omega$, then $2\cdot v=\omega$ so the
dimensions of $C$ are identical with those of $A$.

\begin{theorem}\label{ipreqwipr} Let $u,v\in\ben\cup\{\omega\}$ and 
let $A$ be a $u\times v$ matrix with entries from $\beq$. Define a $u\times (2\cdot v)$ matrix $C$ by, for
$i<u$ and $j<v$, $c_{i,2\cdot j}=a_{i,j}$ and $c_{i,2\cdot j+1}=-a_{i,j}$.
Then
$$\{A\vec x:\vec x\in \beq^v\}=\{C\vec y:\vec y\in(\beq\setminus\{0\})^{2\cdot v}\}\,.$$
Also $A$ is weakly image partition regular over $\beq$ if and only if $C$ is image partition
regular over $\beq$.
\end{theorem}

\begin{proof} Except for the fact that we allow the entries of $A$ to come from 
$\beq$, this is a special case of the following much more general theorem. And the 
proof of that theorem is unchanged as long as multiplication by rationals makes sense.
\end{proof}

The various notions of partition regularity are defined for arbitrary commutative
and cancellative semigroups in exact analogy with Definition \ref{defiprkpr}.

\begin{theorem}\label{ipreqwiprgen} Let $(S,+)$ be a commutative cancellative semigroup
with at least three elements and let $G=S-S$. Let $u,v\in\ben\cup\{\omega\}$ and 
let $A$ be a $u\times v$ matrix with entries from $\bez$. Define a $u\times (2\cdot v)$ matrix $C$ by, for
$i<u$ and $j<v$, $c_{i,2\cdot j}=a_{i,j}$ and $c_{i,2\cdot j+1}=-a_{i,j}$.  Then
$$\{A\vec x:\vec x\in G^v\}=\{C\vec y:\vec y\in(S\setminus\{0\})^{2\cdot v}\}\,.$$
Also $A$ is weakly image partition regular over $S$ if and only if $C$ is image partition
regular over $S$.
\end{theorem}

\begin{proof} Given $\vec y\in(S\setminus\{0\})^{2\cdot v}$,
define $\vec x\in G^v$ by $x_j=y_{2\cdot j}-y_{2\cdot j+1}$ for $j<v$.  Then $A\vec x=C\vec y$.

Note that, since $S$ has at least three elements, for each $x\in G$ there is some $s\in S\setminus\{0\}$
such that $x+s\in S\setminus\{0\}$.  
Given $\vec x\in G^v$, define $\vec y\in(S\setminus\{0\})^{2\cdot v}$ as follows.
Given $j<v$ pick $s_j\in S\setminus\{0\}$ such that $x_j+s_j\in S\setminus\{0\}$,
define $y_{2\cdot j}=s_j+x_j$ and $y_{2\cdot v+1}=s_j$. Then $A\vec x=C\vec y$.  

The fact that $C$ is image partition regular over $S$ if and only if
$A$ is weakly image partition regular over $S$ follows from the displayed 
equation. \end{proof}

The requirement that $|S|\geq 3$ in Theorem \ref{ipreqwiprgen} is needed.
Indeed, if $G=S=\bez_2$ and $A$ is the $2\times 2$ identity matrix, 
then there do not exist $w\in\ben\cup\{\omega\}$
and a $u\times w$ matrix $C$ such that 
$$\{A\vec x:\vec x\in G^2\}=\{C\vec y:\vec y\in(S\setminus\{0\})^w\}$$
because there are four elements of $\{A\vec x:\vec x\in G^2\}$ and only one
element of $\{C\vec y:\vec y\in(S\setminus\{0\})^w\}$.  However, from the
point of view of image partition regularity and weak image partition regularity,
what one cares about is the images all of whose entries are in $S\setminus\{0\}$.
For this if $A$ is a $u\times v$ matrix that is weakly image partition regular
over $S$, and $C$ is the $u\times u$ identity matrix, one does have that
$$\{A\vec x:\vec x\in G^v\}\cap(S\setminus\{0\})^u=\{C\vec y:\vec y\in(S\setminus\{0\})^u\}
\cap(S\setminus\{0\})^u\,,$$
for the trivial reason that $(S\setminus\{0\})^u$ is a singleton.

\section{Duality for finite and infinite matrices}

In this section, we prove several theorems for matrices over $\beq$. Although these are only 
stated for matrices over $\beq$, we observe that the proofs are valid for matrices over arbitrary fields.

\begin{definition}\label{dewfKR} If $u,v\in\ben\cup\{\omega\}$ and if 
$B$ is a $u\times v$ matrix with rational entries, we put 
$K(B)=\{\vec x\in \beq^v:B\vec x=\vec 0\}$, the kernel of $B$, 
and $R(B)=\{B\vec x:\vec x\in\beq^v\}$, the range of $B$.
\end{definition}

In this paper we make use of many standard facts from linear algebra. For the sake of
completeness, we give proofs when these facts may be less well known. For general background
in linear algebra, including a discussion of many related results, see for example \cite{C} or
\cite{Hal}.

\begin{lemma}\label{xinkernel}  Let $u,v\in\ben\cup\{\omega\}$ and let 
$B$ be a $u\times v$ matrix with rational entries. Assume that
the kernel $K=K(B)$ is nontrivial. Let $V$ be the vector space over $\beq$ of all linear
transformations from $K$ to $\beq$. For $i<v$ define $\pi_i\in V$ by, for $\vec x\in K$, 
$\pi_i(\vec x)=x_i$. Let $T\subseteq v$ be maximal subject to the
requirement that $\{\pi_i:i\in T\}$ is linearly independent. 
For each $i<v$ pick $\langle d_{i,j}\rangle_{j\in T}$ in $\beq$
such that $\pi_i=\sum_{j\in T}d_{i,j}\pi_j$.
\begin{itemize}
\item[(1)] If $i,j\in T$, then $d_{i,j}=0$ if $i\neq j$ and $d_{i,j}=1$ if $i=j$.
\item[(2)] If $B$ has no row identically zero, then, for each $k<u$ there is some $l\in v\setminus T$ such that
$b_{k,l}\neq 0$.
\item[(3)]  For each $\vec x\in\beq^v$, the 
following statements are equivalent.
\begin{itemize}
\item[(a)] For all $i\in v$, $x_i=\sum_{j\in T}d_{i,j}x_j$.
\item[(b)] For all $i\in v\setminus T$, $x_i=\sum_{j\in T}d_{i,j}x_j$.
\item[(c)] $\vec x\in K$.
\end{itemize}\end{itemize}\end{lemma}

\begin{proof}  Conclusion (1) is immediate.

For the second conclusion, let $k<u$ and suppose that
for all $l\in v\setminus T$, $b_{k,l}=0$. Then for
each $\vec x\in K$, $0=\sum_{j<v}b_{k,j}x_j=\sum_{j\in T}b_{k,j}x_j
=\sum_{j\in T}b_{k,j}\pi_j(\vec x)$, so
$\overline 0=\sum_{j\in T}b_{k,j}\pi_j$, where $\overline 0$ is the constant
linear transformation.  But  we are assuming that no row of $B$ consists entirely of zeroes. 
So for some $j\in T$, $b_{k,j}\neq 0$ and thus
$\{\pi_i:i\in T\}$ is not linearly independent, a contradiction.

For the third conclusion, let $\vec x\in \beq^v$. The fact that (c) implies (a) is trivial
as is the fact that (a) implies (b).
So assume that for all $i\in v\setminus T$, $x_i=\sum_{j\in T}d_{i,j}x_j$.
By conclusion (1) we have that for all $i<v$, $x_i=\sum_{j\in T}d_{i,j}x_j$.

For any $k<u$ 
$$\textstyle \overline 0=\sum_{i<v}b_{k,i}\pi_i=\sum_{i<v}b_{k,i}\sum_{j\in T}d_{i,j}\pi_j
=\sum_{j\in T}(\sum_{i<v}b_{k,i}d_{i,j})\pi_j\,.$$ Since $\{\pi_i:i\in T\}$ is 
linearly independent, we have that for each $k<u$ and each $j\in T$, $\sum_{i<v}b_{k,i}d_{i,j}=0$.
Therefore $$\textstyle\sum_{i<v}b_{k,i}x_i=\sum_{i<v}b_{k,i}\sum_{j\in T}d_{i,j}x_j=
\sum_{j\in T}(\sum_{i<v}b_{k,i}d_{i,j})x_j=0\,.$$
\end{proof}

\begin{lemma}\label{MatrixC}  Let $u,v\in\ben\cup\{\omega\}$ and let 
$B$ be a $u\times v$ matrix with rational entries. There exists a  
$v\times v$ matrix $C$ with rational entries such that the following statements hold. 
\begin{itemize} 
\item[(1)] For every $\vec x\in \beq^v$, $B\vec x=\vec 0$ if and only if
$C\vec x=\vec x$. 
\item[(2)] $BC$ is the $u\times v$ matrix with all entries equal to $0$.
\item[(3)] $K(B)=R(C)$.
\item[(4)] $C^2=C$. 
\item[(5)] If $S$ is a nontrivial subsemigroup of $\beq$, then the following
statements are equivalent.
\begin{itemize}
\item[(a)] $B$ is kernel partition regular over $S$.
\item[(b)] $C$ is image partition regular over $S$.
\item[(c)] $C$ is weakly image partition regular over $S$.
\end{itemize}\end{itemize}
\end{lemma}

\begin{proof} If the kernel of $B$ is trivial, our lemma holds with $C={\bf O}$, the $v\times v$ matrix with
all entries equal to $0$. So we may suppose that the kernel of $B$ is nontrivial, and hence 
that the hypotheses of Lemma \ref{xinkernel} are satisfied. We define $T$ and $d_{i,j}$ for $i<v$ and $j\in T$ as in 
Lemma \ref{xinkernel}.   Let $C$ be the $v\times v$ matrix such that, for
$i,j<v$, $c_{i,j}=\left\{\begin{array}{cl}d_{i,j}&\hbox{if }j\in T\\
0&\hbox{if }j\notin T\,. \end{array}\right.$
It follows from Lemma \ref{xinkernel}(3) that for each $\vec x\in\beq^v$, $B\vec x=\vec 0$ if and only if $C\vec x=\vec x$. 

Let ${\bf O}$ be the $u\times v$ matrix all of whose entries are $0$.  To see that
$BC={\bf O}$ we need that each column of $C$ is in $K(B)$. If $t\in v\setminus T$,
it is trivial that column $t$ of $C$ is in $K(B)$, so let $t\in T$ and for
$j<v$, let $x_j=c_{j,t}$.  By Lemma \ref{xinkernel}(3), we need to show that
for each $i<v$, $x_i=\sum_{j\in T}c_{i,j}x_j$. So let
$i<v$ be given. Then $\sum_{j\in T}c_{i,j}x_j=\sum_{j\in T}c_{i,j}c_{j,t}=c_{i,t}=x_i$.

Since $BC={\bf O}$, we have $R(C)\subseteq K(B)$. And if $\vec x\in K(B)$, then
$C\vec x=\vec x$ so $\vec x\in R(C)$.  To see that $C^2=C$, let $\vec y\in\beq^v$ and
let $\vec x=C\vec y$.  Then $B\vec x=BC\vec y=\vec 0$ so $C\vec x=\vec x$.

Finally, assume that $S$ is a nontrivial subsemigroup of $\beq$. That (a) implies (b) follows
from conclusion (1) and the fact that (b) implies (c) is trivial. To see that (c) implies (a),
let $S\setminus\{0\}$ be finitely colored and pick $\vec y\in \beq^v$ such that
the entries of $\vec x=C\vec y$ are monochromatic.  By conclusion (4) $C\vec x=\vec x$ so
by conclusion (1), $B\vec x=\vec 0$.
\end{proof}

The first part of our duality is valid for arbitrary nontrivial subsemigroups of $\beq$.
Notice that, as a consequence of this theorem, if $S\setminus\{0\}$ is finitely 
coloured, then the sets of monochromatic kernel elements of $B$ and monochromatic images
of $C$ are equal.

\begin{theorem}\label{BgetsC} Let $u,v\in\ben\cup\{\omega\}$, let $S$ be a nontrivial subsemigroup of $\beq$, and let 
$B$ be a $u\times v$ matrix with rational entries. Then there is a $v\times v$ matrix $C$ with
rational entries such that
$$\{\vec x\in (S\setminus\{0\})^v:B\vec x=\vec 0\}=\{C\vec y:\vec y\in (S\setminus\{0\})^v\}\cap(S\setminus\{0\})^v\,.$$
Further $B$ is kernel partition regular over $S$ if and only if $C$ is image partition regular over $S$.
\end{theorem}

\begin{proof} 
 Let  $C$ be as in Lemma \ref{MatrixC}. For every $x\in \beq^v$, $B\vec x=\vec 0$ implies 
$C\vec x=\vec x$, and so 
$\{\vec x\in (S\setminus\{0\})^v:B\vec x=\vec 0\}\subseteq\{C\vec y:y\in (S\setminus\{0\})^v\}\cap(S\setminus\{0\})^v$.
For the reverse inclusion, let $\vec z\in (S\setminus\{0\})^v$ and assume that $\vec z=C\vec y$ for some $y\in S^v$.  
Then $B\vec z=\vec 0$. The final assertion of the theorem is immediate. \end{proof}

If $v$ is finite, the following lemma is a standard result in linear
algebra. This lemma is \cite[Lemma 3.5]{HSa}, but the proof given here is much
easier. In this lemma, $\bigoplus_{i<v}\beq$ is the direct sum of $v$ copies
of $\beq$, that is members of the Cartesian product $\bigtimes_{i<v}\beq$
with finitely many nonzero entries.

\begin{lemma}\label{span} Let $v\in\ben\cup\{\omega\}$, let $W=\bigoplus_{i<v}\beq$,
let $L\subseteq v$, let $\langle \vec r_i\rangle_{i\in L}$ be a sequence
of linearly independent members of $W$, and let $\langle y_i\rangle_{i\in L}$
be an arbitrary sequence in $\beq$.  Then there exists $\vec x\in\beq^v$ such that
for all $i\in L$, $\vec r_i\cdot\vec x=y_i$.\end{lemma}

\begin{proof} If $v=\omega$, we may presume that $v\setminus L$ is infinite
because then $\bigoplus_{i<v}\beq$ is isomorphic to $\bigoplus_{i<v}\beq\oplus\bigoplus_{i<v}\beq$.
For $i<v$, let $\vec e_i\in W$ be the usual basis vector
defined by 
$$\vec e_i(j)=\left\{\begin{array}{cl}1&\hbox{if }j=i\\
0&\hbox{if }j\neq i\,.\end{array}\right.$$
If $\langle \vec r_i\rangle_{i\in L}$ spans $W$, let $J=L$. Otherwise,
pick $J\subseteq v$ such that $L\subseteq J$ and
pick $\langle \vec r_i\rangle_{i\in J\setminus L}$ such that
$\langle \vec r_i\rangle_{i\in J}$ is a basis for $W$. (If $v=\omega$, this is possible
because $v\setminus L$ is infinite.)

Define a linear transformation $f:W\to\beq$ by, for $i\in J$,
$$f(\vec r_i)=\left\{\begin{array}{cl}y_i&\hbox{if }i\in L\\
0&\hbox{if }i\in J\setminus L\,.\end{array}\right.$$
For each $i\in v$ let $x_i=f(\vec e_i)$. Define a linear
transformation $g:W\to\beq$ by, for $\vec z\in W$,
$g(\vec z)=\vec z\cdot\vec x$.  Then for each
$i<v$, $g(\vec e_i)=\vec e_i\cdot \vec x=x_i=f(\vec e_i)$.
Since $f$ and $g$ agree on a basis for $W$, they are equal.
Therefore, for $i\in L$, $y_i=f(\vec r_i)=g(\vec r_i)=\vec r_i\cdot\vec x$.
\end{proof}

For finite matrices, the matrix $B$ in the following lemma was introduced in \cite[Theorem 2.2]{HL}.

\begin{lemma}\label{matrixB}  Let $u,v\in\ben\cup\{\omega\}$ and let 
$A$ be a $u\times v$ matrix with rational entries such that the rows of 
$A$ are linearly dependent over $\beq$. For $i<u$, let $\vec r_i$ be the $i^{\hbox{\smallrm th}}$
row of $A$.  Let $L$ be a subset of $u$ that is maximal with respect to the property
that $\{\vec r_i:i\in L\}$ is linearly independent without repeated rows and let 
$J=u\setminus L$. For $i\in J$, let $\langle b_{i,t}\rangle_{t\in L}$ be the
rational numbers such that $\vec r_i=\sum_{t\in L}b_{i,t}\vec r_t$ and for
$i$ and $t$ in $J$, let $b_{i,t}=0$ if $i\neq t$ and $b_{i,t}=-1$ if $i=t$.
Let $B$ be the $J\times u$ matrix whose entry in row $i$ and column $t$ is
$b_{i,t}$. Then $BA={\bf O}$, where ${\bf O}$ is the $J\times v$ matrix with all
zero entries. Furthermore, $K(B)=R(A)$.
\end{lemma}

\begin{proof} For every $i\in J$, the $i^{\hbox{\smallrm th}}$ row of $BA$ is
$-\vec r_i+\sum_{t\in L}b_{i,t}\vec r_t=\vec 0$, and so $BA={\bf O}$.

Clearly, $R(A)\subseteq K(B)$.
To see that $K(B)\subseteq R(A)$, let $\vec y\in K(B)$. Then, for every $i\in J$, 
$y_i=\sum_{t\in L}b_{i,t}y_t$. By Lemma \ref{span}, there exists $\vec x\in \beq^v$
such that $\vec r_i\cdot \vec x=y_i$ for every $i\in L$. Hence, for every $i\in J$,
$\vec r_i\cdot \vec x=(\sum_{t\in L} b_{i,t}\vec r_t)\cdot\vec x=\sum_{t\in L}b_{i,t}y_t=y_i$ and so $\vec y=A\vec x\in R(A)$.
\end{proof}

Notice that if the rows of $A$ are linearly independent and $B$ is any matrix with
$u$ columns and all entries equal to $0$, then all entries of $BA$ are $0$ and 
$K(B)=R(A)$.

\begin{corollary}\label{stronglyIPR}  Let $u,v\in\ben\cup\{\omega\}$, let $S$ be a nontrivial subsemigroup of $\beq$, and let 
$A$ be a $u\times v$ matrix with rational entries that is weakly image partition regular over $S$. Then there is
a $u\times u$ matrix $C$ with rational entries that is image partition regular over $S$ for which $R(C)=R(A)$ and $C^2=C$.
\end{corollary}
\begin{proof} If the rows of $A$ are linearly independent, we may let $C$ be the 
$u\times u$ identity matrix, so assume that the rows of $A$ are linearly dependent.
Let $B$ be the matrix associated with $A$ by Lemma \ref{matrixB}, and let
$C$ be the matrix associated with $B$ by Lemma \ref{MatrixC}. Then $B$ is kernel partition regular over $S$
and $C$ is image partition regular over $S$. Since $R(A)=K(B)$ and $R(C)=K(B)$, it follows that
$R(C)=R(A)$.  
\end{proof}

The next theorem is the second half of our duality.

\begin{theorem}\label{AgetsB} 
  Let $u,v\in\ben\cup\{\omega\}$, let 
$A$ be a $u\times v$ matrix with rational entries.  Then there exist $J\subseteq u$ and a $J\times u$ matrix
$B$ such that  
$$\{\vec y\in (\beq\setminus\{0\})^u:B\vec y=\vec 0\}=
\{A\vec x:\vec x\in\beq^v\}\cap(\beq\setminus\{0\})^u\,.$$
Further, $A$ is weakly image partition regular over $\beq$ if and only if $B$ is kernel
partition regular over $\beq$.\end{theorem}

\begin{proof} 
If the rows of $A$ are linearly independent, we may let $B$ be the $u\times u$ matrix with all
zero entries, so
assume that the rows of $A$ are linearly dependent.
Let $J$ and $B$ be as in Lemma \ref{matrixB}.  Since $K(B)=R(A)$, the conclusions follow.\end{proof}

In contrast with Theorem \ref{BgetsC}, we see that there is no nontrivial proper
subgroup of $\beq^+=\{x\in\beq:x>0\}$ to which Theorem \ref{AgetsB} extends, even if the
hypothesis is strengthened to require $A$ to be image partition regular, not
just weakly image partition regular.   

\begin{theorem}\label{notG} Let $S$ be a nontrivial proper subsemigroup of
$\beq^+$ or a nontrivial proper subgroup of $\beq$.  Then there is a $3\times 2$ matrix $A$ that is image partition regular
over $S$ (whose rows are necessarily linearly dependent) and has the property that 
there do not exist $k\in\ben$ and a $k\times 3$ matrix $B$
such that 
$$\{\vec y\in (S\setminus\{0\})^3:B\vec y=\vec 0\}=
\{A\vec x:\vec x\in S^2\}\cap(S\setminus\{0\})^3\,.$$
In fact, there do not exist $k\in\ben$ and a $k\times 3$ matrix $B$
such that 
$$\begin{array}{rcl}\{A\vec x:\vec x\in (S\setminus\{0\})^2\}\cap(S\setminus\{0\})^3&\subseteq
&\{\vec y\in (S\setminus\{0\})^3:B\vec y=\vec 0\}\\&\subseteq&
\{A\vec x:\vec x\in S^2\}\cap(S\setminus\{0\})^3\,.\end{array}$$ \end{theorem}

\begin{proof} If $S$ is a nontrivial proper subgroup of $\beq$, then
$\{x\in S:x>0\}$ is a nontrivial proper subsemigroup of $\beq^+$ so $\{x\in\beq:x>0\}\setminus S\neq\emp$.

Assume first that $1\notin S$ and let
$d=\min(S\cap \ben)$.  Let 
$A=\left(\begin{array}{cc} d&0\\ 0&d\\
d&d\end{array}\right)\,.$
To see that $A$ is image partition regular over $S$, let 
$\varphi$ be a finite colouring of $S$ and define a finite colouring $\psi$
of $\ben$ by, for $x\in\ben$, $\psi(x)=\varphi(d^2x)$.  Pick by Schur's 
Theorem $x_0$ and $x_1$ in $\ben$ such that $\psi(x_0)=\psi(x_1)=\psi(x_0+x_1)$.
Then $\vec x=\left(\begin{array}{c}dx_0\\ dx_1\end{array}\right)\in (S\setminus\{0\})^2$
and $A\vec x$ is monochromatic with respect to $\varphi$.

Suppose we have $k\in\ben$ and a $k\times 3$ matrix $B$
such that $$\begin{array}{rcl}\{A\vec x:\vec x\in (S\setminus\{0\})^2\}\cap(S\setminus\{0\})^3&\subseteq&
\{\vec y\in (S\setminus\{0\})^3:B\vec y=\vec 0\}\\&\subseteq&
\{A\vec x:\vec x\in S^2\}\cap(S\setminus\{0\})^3\,.\end{array}$$

Since $A\left(\begin{array}{c}d\\ d\end{array}\right)=
\left(\begin{array}{c}d^2\\ d^2\\ 2d^2\end{array}\right)$, we have
that $B\left(\begin{array}{c}d^2\\ d^2\\ 2d^2\end{array}\right)=\vec 0$
and therefore $B\left(\begin{array}{c}d\\ d\\ 2d\end{array}\right)=\vec 0$.
Consequently there is some $\vec x\in S^2$ such that
$A\vec x=\left(\begin{array}{c}d\\ d\\2d \end{array}\right)$.
But then $\vec x=\left(\begin{array}{c}1\\ 1\end{array}\right)\notin S^2$.

Now assume that $1\in S$ and pick some $d\in\ben$ such that $\frac{1}{d}\notin S$. Let
$A=\left(\begin{array}{cc} 0&1\\ d&1\\d&2\end{array}\right)\,.$
Since $A$ is a first entries matrix, by \cite[Theorem 3.1]{HL} $A$ is image partition regular over $\ben$ and therefore
over $S$.
Suppose we have $k\in\ben$ and a $k\times 3$ matrix $B$
such that $$\begin{array}{rcl}\{A\vec x:\vec x\in (S\setminus\{0\})^2\}\cap(S\setminus\{0\})^3&\subseteq&
\{\vec y\in (S\setminus\{0\})^3:B\vec y=\vec 0\}\\&\subseteq&
\{A\vec x:\vec x\in S^2\}\cap(S\setminus\{0\})^3\,.\end{array}$$

Since $A\left(\begin{array}{c}1\\ d\end{array}\right)=
\left(\begin{array}{c}d\\ 2d\\ 3d\end{array}\right)$, we have
that $B\left(\begin{array}{c}d\\ 2d\\ 3d\end{array}\right)=\vec 0$
and therefore $B\left(\begin{array}{c}1\\ 2\\ 3\end{array}\right)=\vec 0$.
Consequently there is some $\vec x\in S^2$ such that
$A\vec x=\left(\begin{array}{c}1\\ 2\\ 3\end{array}\right)$.
But then $\vec x=\left(\begin{array}{c}1/d\\ 1\end{array}\right)\notin S^2$.
\end{proof}  

It is well known, and was essentially a part of Deuber's proof of Rado's
conjecture, that if $u,v\in\ben$ and $A$ is a $u\times v$ matrix that is 
weakly image partition regular over $\bez$, then 
there is a matrix $D$ that is kernel partition
regular over $\bez$ such that if $\vec y$ has nonzero
entries and $D\vec y=\vec 0$, then the set of entries of $\vec y$ contains
the entries of $A\vec x$ for some $\vec x\in \bez^v\setminus\{\vec 0\}$.

We show in Theorem \ref{imgcontained} that the corresponding statement is
true for a finite matrix that is weakly image partition regular over any nontrivial
subgroup of $(\beq,+)$.

\begin{lemma}\label{lemDkpr} Let $S$ be a nontrivial subgroup of $\beq$, let $j,u\in\ben$, 
let $B$ be a $j\times u$ matrix with rational entries that is kernel partition regular 
over $\beq$, let ${\bf O}$ be the $j\times u$ matrix with all entries equal to $0$,
let $I$ be the $u\times u$ identity matrix, let $c\in S$, and let
$$D=\left(\begin{array}{ccc}B&{\bf O}&{\bf O}\\
I&cI&-cI\end{array}\right)\,.$$
Then $D$ is kernel partition regular over $S$.\end{lemma}

\begin{proof} Since $B$ is kernel partition
regular over $\beq$, it satisfies the columns condition.  Let $m$ and a partition
$\langle I_t\rangle_{t=0}^{m-1}$ of $\ohat{u-1}$ be as guaranteed by the columns
condition for $B$.  

Let $I_0'=\{u,u+1,\ldots,3u-1\}$ and for $t\in\nhat{m}$, let $I_t'=I_{t-1}$.
Then $\langle I_t'\rangle_{t=0}^{m}$ is a partition of $\ohat{3u-1}$. It is
routine to verify that with this partition, $D$ satisfies the columns condition.
Therefore, $D$ is kernel partition regular over $\bez$.  Let $d=\min(S\cap \ben)$.
To see that $D$ is kernel partition regular over $S$, let $\varphi$ be a
finite colouring of $S\setminus\{0\}$.  Define $\psi$ on $\bez$ by
$\psi(x)=\varphi(dx)$. If $\vec x$ is monochromatic with respect to 
$\psi$ and $D\vec x=\vec 0$, then $d\vec x$ is monochromatic with respect
to $\varphi$ and $Dd\vec x=\vec 0$.
\end{proof}

\begin{theorem}\label{imgcontained} Let $S$ be a nontrivial proper subgroup of
$\beq$, let $u,v\in\ben$, and let $A$ be a $u\times v$  matrix with rational entries
that is weakly image partition regular over $S$ and has linearly dependent rows.
Then there exists $j<u$ and a $(j+u)\times 3u$ matrix $D$ that is kernel partition
regular over $S$ such that, whenever $\vec s\in (S\setminus\{0\})^{3u}$ and
$D\vec s=\vec 0$, there exists $\vec y\in S^v\setminus\{\vec 0\}$ such that
all entries of $A\vec y$ are entries of $\vec s$.
\end{theorem}

\begin{proof} We first show that it suffices to assume that the entries of 
$A$ are in $S\cap\bez$. To see this pick $d\in\ben$ such that all entries
of $dA$ are in $S\cap\bez$. (One can take $d$ to be the product of a common multiple
of the denominators of entries of $A$ with the minimum of $S\cap \ben$.)
Then $dA$ is weakly image partition regular over $S$. (If $\varphi$ is a finite
colouring of $S\setminus\{0\}$, we define $\psi$ on $S\setminus\{0\}$
by $\psi(x)=\varphi(dx)$. If the entries of $A\vec x$ are monochromatic
with respect to $\psi$, then the entries of $dA\vec x$ are monochromatic
with respect to $\varphi$.)  Assume that we have
$j<u$ and a $(j+u)\times 3u$ matrix $D$ that is kernel partition
regular over $S$ such that, whenever $\vec s\in (S\setminus\{0\})^{3u}$ and
$D\vec s=\vec 0$, there exists $\vec y\in S^v\setminus\{\vec 0\}$ such that
all entries of $dA\vec y$ are entries of $\vec s$. If the entries
of $dA\vec y$ are entries of $\vec s$ and $\vec z=d\vec y$, then the entries
of $A\vec z$ are entries of $\vec s$.

Thus we assume that the entries of $A$ are in $S\cap\bez$.  Let $l=\rank(A)$.
By rearranging rows and columns, we may assume that the upper left 
$l\times l$ corner $A^*$ of $A$ has linearly independent rows. 
Let $c$ be the absolute value of $\det(A^*)$. 

Since $A$ is weakly image partition regular over $S$ it is weakly image partition
regular over $\beq$.  Let $B$ be as in Lemma \ref{matrixB} with $L=\ohat{l-1}$ so that
$J=\{l,l+1,\ldots,u-1\}$.  Let $j=u-l$ and let $B'$ be the $j\times u$ matrix
whose entry in row $t$ and column $i$ is $b_{t,i}'=b_{l+t,i}$.  Since
$B'A={\bf O}$, we have that $B'$ is kernel partition regular over $\beq$.

Let ${\bf O}$ be the $j\times u$ matrix with all entries equal to $0$,
let $I$ be the $u\times u$ identity matrix, and let
$$D=\left(\begin{array}{ccc}B'&{\bf O}&{\bf O}\\
I&cI&-cI\end{array}\right)\,.$$
By Lemma \ref{lemDkpr}, $D$ is kernel partition regular over $S$.

Let $\vec s\in (S\setminus\{0\})^{3u}$ such that
$D\vec s=\vec 0$.  Let $\vec x\in (S\setminus\{0\})^{u}$ and
$\vec r\in (S\setminus\{0\})^{2u}$ such that $\vec s=\left(\begin{array}{c} \vec x\\
\vec r\end{array}\right)$. Then for $i\in\ohat{u-1}$,
$x_i=c(r_{u+i}-r_{i})$.  Letting $z_i=r_{u+i}-r_{i}$ we have
that each $z_i\in S$. Let $\vec z\,'$ be the first $l$ entries of $\vec z$. 
Let $\vec w$ be the member of $\beq^l$ such that
$A^*\vec w=\vec z\,'$.  Then by Cramer's rule, we have
that for each $i\in\ohat{l-1}$, $cw_i\in S$.

For $i\in\ohat{l-1}$, let $y_i=cw_i$ and for
$i\in\{l,l+1,\ldots,v\}$, let $y_i=0$.  
Then a routine computation establishes that $A\vec y=\vec x$.
\end{proof}

We ask now whether some version of Theorem \ref{imgcontained} applies
to matrices with infinitely many rows.  (Our proof of Lemma \ref{lemDkpr}
uses Rado's Theorem, so is inherently finite.)

\begin{question}\label{qimage} Let $v\in\ben\cup\{\omega\}$ and
let $A$ be an $\omega\times v$ matrix with rational entries.
Is it true that for some or all nontrivial proper subgroups $S$ of $\beq$,
if $A$ is weakly image partition regular over $S$, then there
is a kernel partition regular matrix $D$ such that whenever
$\vec s$ is in the kernel of $D$ with entries in $S\setminus \{0\}$
there must exist $\vec y\in S^v$ such that all entries of $A\vec y$ are
entries of $\vec s$? \end{question} 

We note that if $\beq$ is given any reasonable topology (more
precisely, any topology in which $0$ has a neighbourhood distinct from $\beq$)
and $\beq^u$ and $\beq^v$ have
the product topology, it is routine to establish that for any $u,v\in\ben\cup\{\omega\}$, the 
$u\times v$ matrices with finitely many non-zero entries in each row 
correspond precisely to the continuous linear transformations from $\beq^v$ to $\beq^u$.
(One uses the fact that for each $a\in \beq\setminus\{0\}$, there is some
$t\in\beq$ such that $at\notin U$.)

We also remark that the theorems in this section, when proved in the context of an arbitrary
field $F$, provide us with information about continuous linear maps from $F^v$ to $F^u$. 
For example, Lemma \ref{matrixB} implies that the range of every such map is closed.

\section{Image partition regularity and weak image\\ partition regularity over $\beq$}

We saw in Theorem \ref{ipreqwipr} that for any commutative cancellative
semigroup $S$ with at least three elements and any $u\times v$ matrix 
$A$ that is weakly image partition regular over $S$ there is a
$u\times (2\cdot v)$ matrix that is image partition regular over $S$ and
has exactly the same images as $A$. The following corollary of the 
results of Section 2 shows that for $S=\beq$, one need not add
additional columns to at least get the same images with entries in 
$S\setminus\{0\}$.

\begin{corollary}\label{AgetsC}   Let $u,v\in\ben\cup\{\omega\}$ and let 
$A$ be a $u\times v$ matrix with rational entries that is weakly image partition regular over $\beq$.
Then there exist a set $T\subseteq u$ and a $u\times T$ matrix $D$ that is image partition regular over
$\beq$ such that $|T|\leq v$ and
$$\{A\vec x:\vec x\in\beq^v\}\cap(\beq\setminus\{0\})^u=
\{D\vec y:\vec y\in (\beq\setminus\{0\})^T\}\cap(\beq\setminus\{0\})^u\,.$$
\end{corollary}

\begin{proof} Assume first that the rows of $A$ are linearly independent. Then $u\leq v$ and
by Lemma \ref{span},
$\{A\vec x:\vec x\in\beq^v\}\cap(\beq\setminus\{0\})^u=(\beq\setminus\{0\})^u$
so one may let $C$ be the $u\times u$ identity matrix.

So assume that the rows of $A$ are linearly dependent over $\beq$.
By Theorem \ref{AgetsB} pick 
$J\subseteq u$ and a $J\times u$ matrix
$B$ such that $B$ is kernel partition regular over $\beq$ and 
$\{\vec w\in (\beq\setminus\{0\})^u:B\vec w=\vec 0\}=
\{A\vec x:\vec x\in\beq^v\}\cap(\beq\setminus\{0\})^u$.
Since $B$ is kernel partition regular over $\beq$, the kernel
of $B$ is nontrivial so the proof of Lemma \ref{MatrixC} produces 
a $u\times u$ matrix $C$, a set $T\subseteq u$, and $d_{i,j}$ in $\beq$
for $i<u$ and $j\in T$ such that $|T|$ is the dimension of the
kernel of $B$ and for $i,j<u$, $c_{i,j}=\left\{\begin{array}{cl}d_{i,j}&\hbox{if }j\in T\\
0&\hbox{if }j\notin T\,. \end{array}\right.$  Let $D$ be the $u\times T$ 
matrix with entries $d_{i,j}$ for $i<u$ and $j\in T$.  As we saw in the proof of 
Theorem \ref{BgetsC}, 
$$\{\vec x\in (\beq\setminus\{0\})^u:B\vec x=\vec 0\}=\{C\vec y:\vec y\in (\beq\setminus\{0\})^u\}\cap(\beq\setminus\{0\})^u\,,$$
and trivially $\{C\vec y:\vec y\in (\beq\setminus\{0\})^u\}=\{D\vec y:\vec y\in (\beq\setminus\{0\})^T\}$.

Checking the proof of Theorem \ref{AgetsB}, we see that $J=u\setminus L$ where $|L|$ is the
rank of $A$.  And we also see that $B$ is of the form $\left(\begin{array}{cc}E&-I_J\end{array}\right)$,
where $E$ is a $J\times L$ matrix and $I_J$ is the $J\times J$ identity matrix so that the
rank of $B$ is $|J|$.  And we have seen that $|T|$ is the 
dimension of the kernel of $B$, which is $|u\setminus J|=|L|\leq v$.
\end{proof}

We saw in Theorem \ref{notG} that Theorem \ref{AgetsB} does not
extend to any proper subsemigroup of $\beq^+$ or any proper subgroup
of $\beq$.  

In Theorem \ref{notIPRG} we will show that Corollary \ref{AgetsC} does not
extend to any proper subgroup of $\beq$.

\begin{lemma}\label{kcl} Let $G$ be a proper subgroup of $\beq$.  If 
$1\in G$, pick $d\in \ben$ such that $\frac{1}{d}\notin G$ and let
$p=1$. If $1\notin G$, pick $d\in\ben\cap G$ and let $p=d$.
Let $A=\left(\begin{array}{cc}1&0\\ 1&d\\ 1&2d\\ \vdots&\vdots\end{array}\right)$.
Assume that $C$ is an $\omega\times 2$ matrix with rational entries such that
$\{A\vec x:\vec x\in G^2\}\cap(G\setminus\{0\})^\omega\subseteq\{C\vec y:y\in(G\setminus\{0\})^2\}$.
Then there exist $x_0,x_1$ in $G\setminus\{0\}$, $m\neq n$ in $\omega$, and $k\in\beq$
such that $k=dx_0x_1(m-n)$ and for all $l<\omega$,
$kc_{l,0}=pdx_1(l-n)$ and $kc_{l,1}=pdx_0(m-l)$.\end{lemma}

\begin{proof} We have that $\overline p=A\left(\begin{array}{c} p\\ 0\end{array}\right)$, where
$\overline p$ is the vector all of whose entries are $p$. So pick 
$x_0$ and $x_1$ in $G\setminus\{0\}$ such that for all $l<\omega$, 
$x_0c_{l,0}+x_1c_{l,1}=p$. Also $A\left(\begin{array}{c} p\\ p\end{array}\right)\in
\{A\vec z:\vec z\in G^2\}\cap (G\setminus\{0\})^\omega$ so pick
$y_0$ and $y_1$ in $G\setminus\{0\}$ such that for all $l<\omega$, 
$y_0c_{l,0}+y_1c_{l,1}=p+lpd$.

Let $k=x_1y_0-x_0y_1$.  Then for each $l<\omega$,
$kc_{l,0}=x_1(p+pld-y_1c_{l,1})-x_0y_1c_{l,0}=p(x_1-y_1)+x_1pld$ and
$kc_{l,1}=p(y_0-x_0)-x_0pld$.  Since the values change with $l$, we
have that $k\neq 0$.

We claim that there is some $m<\omega$ such that
$y_0=(1+md)x_0$.  To see this, suppose not and let
$a=x_0-y_0$ and $b=x_0$. Then for $l<\omega$, entry $l$ of
$A\left(\begin{array}{c}a\\ b\end{array}\right)$ is
$x_0-y_0+ldx_0\in G\setminus\{0\}$, so pick
$u$ and $v$ in $G\setminus\{0\}$ such that
$A\left(\begin{array}{c}a\\ b\end{array}\right)=C\left(\begin{array}{c}u\\ v\end{array}\right)$.
Then $kA\left(\begin{array}{c}a\\ b\end{array}\right)=kC\left(\begin{array}{c}u\\ v\end{array}\right)$
so $ka=p(x_1-y_1)u+p(y_0-x_0)v$ and $ka+kbd=p(x_1-y_1)u+x_1pdu+p(y_0-x_0)v-x_0pdv$.
Solving these equations for $u$ and $v$ we get $u=\frac{1}{p}(ax_0+by_0-bx_0)$ and
$v=\frac{1}{p}(ax_1+by_1-bx_1)$.  But then
$u=\frac{1}{p}(x_0(x_0-y_0)+(y_0-x_0)x_0)=0$, a contradiction.
Similarly, we have some $n<\omega$ such that $y_1=(1+nd)x_1$.

Now $k=x_1y_0-y_1x_0=dx_0x_1(m-n)$ and since $k\neq 0$, $m\neq n$.
Also for $l<\omega$, $kc_{l,0}=p(x_1-y_1)+pldx_1=pdx_1(l-n)$ and $kc_{l,1}=pdx_0(m-l)$.
\end{proof}

\begin{theorem}\label{notIPRG} Let $G$ be a proper subgroup of $\beq$.  There is an
$\omega\times 2$ matrix $A$ that is weakly image partition regular over $G$, but there
does not exist an $\omega\times 2$ matrix $C$ with rational entries such that
$$\{A\vec x:\vec x\in G^2\}\cap(G\setminus\{0\})^\omega=
\{C\vec y:\vec y\in (G\setminus\{0\})^2\}\cap(G\setminus\{0\})^\omega\,.$$
\end{theorem}

\begin{proof} If 
$1\in G$, pick $d\in \ben$ such that $\frac{1}{d}\notin G$ and let
$p=1$. If $1\notin G$, pick $d\in\ben\cap G$ and let $p=d$.
Let $A=\left(\begin{array}{cc}1&0\\ 1&d\\ 1&2d\\ \vdots&\vdots\end{array}\right)$.
Note that $A$ is weakly image partition regular over $G$ since
$A\left(\begin{array}{c}d\\ 0\end{array}\right)=\overline d$.
Suppose we have an $\omega\times 2$ matrix $C$ with rational entries such that
$$\{A\vec x:\vec x\in G^2\}\cap(G\setminus\{0\})^\omega=
\{C\vec y:\vec y\in (G\setminus\{0\})^2\}\cap(G\setminus\{0\})^\omega\,.$$
Pick $x_0$, $x_1$, $m$, $n$, and $k$ as guaranteed by Lemma \ref{kcl}.
Let $u=(m+1)x_0$ and $v=(n+1)x_1$.  Then given $l<\omega$, entry $l$ of 
$kC\left(\begin{array}{c}u\\ v\end{array}\right)$ is
$kc_{l,0}u+kc_{l,1}v=pdx_0x_1(m-n)(l+1)=kp(l+1)$ so entry $l$ of
$C\left(\begin{array}{c}u\\ v\end{array}\right)$ is $p(l+1)\in G\setminus\{0\}$.
Pick $a$ and $b$ in $G$ such that
$A\left(\begin{array}{c}a\\ b\end{array}\right)=C\left(\begin{array}{c}u\\ v\end{array}\right)$.
Then $a=p$ and $a+db=2p$ so $db=p$.  If $p=1$, this says $b=\frac{1}{d}\notin G$.
If $p=d$, this says $b=1\notin G$.
\end{proof}
 
We do not know whether the result of Corollary \ref{AgetsC} extends to 
arbitrary proper subsemigroups of $\beq^+$, or even whether it extends to $\ben$.
We know that the proof of Theorem \ref{notIPRG} does not work to show that it does
not extend to $\ben$ and the proof of Corollary \ref{AgetsC} does not work
to show that it does extend to $\ben$.  

To see that the proof of Theorem \ref{notIPRG} does not work for $\ben$, 
let $d\in\ben\setminus\{1\}$ so that 
$$A=\left(\begin{array}{cc}1&0\\
1&d\\
1&2d\\
\vdots&\vdots\end{array}\right)\,.$$
Let 
$$C=\left(\begin{array}{cc}1&-1\\
1-d&-1+2d\\
1-2d&-1+4d\\
\vdots&\vdots\end{array}\right)\,.$$
Then 
$\{A\vec x:\vec x\in \bez^2\}\cap\ben^\omega=
\{C\vec y:\vec y\in \ben^2\}\cap\ben^\omega$. (If $A\left(\begin{array}{c}a\\ b\end{array}\right)\in\ben^\omega$,
then $a\in\ben$, $b\in\omega$, and
$A\left(\begin{array}{c}a\\ b\end{array}\right)=C\left(\begin{array}{c}2a+b\\ a+b\end{array}\right)$. And for any
$u$ and $v$ in $\ben$,
$A\left(\begin{array}{c}u-v\\ 2v-u\end{array}\right)=C\left(\begin{array}{c}u\\ v\end{array}\right)$.)

To see that the proof of Corollary \ref{AgetsC} does not work for $\ben$,
let
$$A=\left(\begin{array}{cc}1&0\\
1&2\\
1&4\\
\vdots&\vdots\end{array}\right)\,.$$
 The matrix $B$ produced by
 Lemma \ref{matrixB} is
$$B=\left(\begin{array}{ccccccc}
-1&2&-1&0&0&0&\ldots\\
-2&3&0&-1&0&0&\ldots\\
-3&4&0&0&-1&0&\ldots\\
\vdots&\vdots&\vdots&\vdots&\vdots&\vdots&\ddots
\end{array}\right)\,.$$
This matrix is kernel partition regular over $\ben$ since $B\overline 1=\overline 0$.
Then the matrix $C$ produced by Theorem \ref{BgetsC} is
$$C=\left(\begin{array}{cc} 1&0\\
0&1\\
-1&2\\
-2&3\\
-3&4\\
\vdots&\vdots\end{array}\right)\,.$$
If $A\vec x=C\left(\begin{array}{c}1\\ 2\end{array}\right)$,
then $x_0=1$ and $x_1=\frac{1}{2}$.

\begin{question}\label{qanysg} Let $S$ be a nontrivial proper subsemigroup of $\beq^+$ and
let $G$ be the subgroup of $\beq$ generated by $S$. Let
$u,v\in\ben\cup\{\omega\}$ and let 
$A$ be a $u\times v$ matrix with rational entries that is weakly image partition regular over $S$.
Must there exist $w\leq v$ and a $u\times w$ matrix $C$ such that
$$\{A\vec x:\vec x\in G^v\}\cap(S\setminus\{0\})^u=
\{C\vec y:\vec y\in (S\setminus\{0\})^w\}\cap(S\setminus\{0\})^u\,?$$
\end{question}

A matrix produced in answer to this question would necessarily be image partition regular over $S$.

Notice that, by Theorem \ref{ipreqwipr}, Question \ref{qanysg} has an affirmative answer if
$v=\omega$, so we are concerned with Question \ref{qanysg} when $v\in\ben$.  
In Theorem \ref{almost} we give a partial answer, not for every subsemigroup of $\beq^+$, but
for subsemigroups of $\beq^+$ that are the intersection of the group they generate with $\beq^+$.
This partial answer is strong enough
for us to conclude that for any matrix $A$ that is weakly image partition regular over such a semigroup, there 
is an image partition regular matrix $C$ every relevant image of which is an image of $A$.

We recall the definition of the sign function.

\begin{definition}\label{defsgn} Let $x\in\ber$.
$$\sgn(x)=\left\{\begin{array}{cl}1&\hbox{if }x>0\\
0&\hbox{if }x=0\\
-1&\hbox{if }x<0\,.\end{array}\right.$$
\end{definition}

\begin{lemma}\label{onezerominus} Let $S$ be a nontrivial proper subsemigroup of $\beq$ and
let $G$ be the subgroup of $\beq$ generated by $S$, let
$u\in\ben\cup\{\omega\}$, let $v\in\ben$, and let 
$A$ be a $u\times v$ matrix with rational entries that is weakly image partition regular over $S$.
Then there exists $\vec a\in\{-1,0,1\}^v$ such that whenever $S\setminus\{0\}$ is finitely coloured, there exists
$\vec x\in G^v$ such that the entries of $A\vec x$ are monochromatic and for every $i<v$,
$\sgn(x_i)=a_i$.  Further $\vec a\neq\vec 0$.
\end{lemma}

\begin{proof} Suppose not and for each $\vec a\in\{-1,0,1\}^v$ pick a finite colouring
$\varphi_{\vec a}$ of $S\setminus\{0\}$ such that whenever $\vec x\in G^v$ and the entries of $A\vec x$ are
monochromatic with respect to $\varphi_{\vec a}$, there is some $i<v$ such that $\sgn(x_i)\neq a_i$.
Let $\psi$ be a finite coloring of $S\setminus\{0\}$ such that for $x,y\in S$, $\psi(x)=\psi(y)$ 
if and only if for every $\vec a\in\{-1,0,1\}^v$, $\varphi_{\vec a}(x)=\varphi_{\vec a}(y)$.
Pick $\vec x\in G^v$ such that the entries of $A\vec x$ are monochromatic with respect to 
$\psi$. Letting $\vec a_i=\sgn(x_i)$ for each $i<v$ yields a contradiction. Trivially 
$\vec a\neq\vec 0$.
\end{proof}

\begin{theorem}\label{almost}  Let $G$ be a nontrivial subgroup of $\beq$, let
$S=G\cap\beq^+$, let
$u\in\ben\cup\{\omega\}$, let $v\in\ben$, and let 
$A$ be a $u\times v$ matrix with rational entries that is weakly image partition regular over $S$.
Then there exists a $u\times v$ matrix $C$ with rational entries (whose entries are in the subgroup of
$\beq$ generated by the entries of $A$) such that
\begin{itemize}
\item[(1)] $\{C\vec y:\vec y\in S^v\}\subseteq \{A\vec x:x\in G^v\}$ and
\item[(2)] whenever $S$ is finitely coloured, there exists $\vec x\in G^v$ such that
the entries of $A\vec x$ are monochromatic and there is some $\vec y\in S^v$ 
such that $C\vec y=A\vec x$.
\end{itemize}
\end{theorem}

\begin{proof} Pick $\vec a$ as guaranteed by Lemma \ref{onezerominus}. Since
$\vec a\neq \vec 0$ we may presume (by permuting the columns of $A$) that
$a_0\neq 0$. Define a $v\times v$ matrix $E$ as follows.  For
$j<v$, $e_{j,0}=a_0$ and for $i\in\nhat{v-1}$, let $e_{0,i}=0$.  For $j\in\nhat{v-1}$ and $i<v$, 
$$e_{j,i}=\left\{\begin{array}{cl} 1&\hbox{if }i=j\hbox{ and }a_i\geq 0\\
-1&\hbox{if }i=j\hbox{ and }a_i< 0\\
0&\hbox{otherwise.}\end{array}\right.$$
Then $E$ is a lower triangular matrix whose diagonal entries are all
$1$ or $-1$ so $\det(A)=\pm 1$.  Thus $E$ is invertible and its 
inverse has integer entries.  (In fact, all entries of $E^{-1}$ are $0$, $1$, or
$-1$.)  Let $C=AE^{-1}$.  Conclusion (1) is immediate.

To verify (2), let $S$ be finitely coloured and pick $\vec x\in G^v$
such that the entries of $A\vec x$ are monochromatic and for every
$i<v$, $a_i=\sgn(x_i)$.  Let $y=E\vec x$. Then $C\vec y=A\vec x$ so
it suffices to show that the entries of $\vec y$ are positive.
We have that $y_0=a_0x_0>0$.  For $j>0$, $y_j=a_0x_0$ if $x_j=0$ and
otherwise $y_j=a_0x_0+a_jx_j>0$.
\end{proof}

\begin{corollary}\label{wiprIPR}  Let $G$ be a nontrivial subgroup of $\beq$, let
$S=G\cap\beq^+$, let
$u,v\in\ben\cup\{\omega\}$, and let 
$A$ be a $u\times v$ matrix with rational entries. Then $A$ is
weakly image partition regular over $S$ if and only if there
is a $u\times v$ matrix $C$ with rational entries that is
image partition regular over $S$ and for every $\vec y\in S^v$,
there exists $\vec x\in G^v$ such that $A\vec x=C\vec y$.
\end{corollary}

\begin{proof} The sufficiency is trivial. For the necessity we can
invoke Theorem \ref{ipreqwipr} if $v=\omega$, so assume that
$v\in\ben$.  Pick $C$ as guaranteed by Theorem \ref{almost}.
One has immediately that for every $\vec y\in S^v$,
there exists $\vec x\in G^v$ such that $A\vec x=C\vec y$.

To see that $C$ is image partition regular over $S$, let
$S$ be finitely coloured. Pick $\vec x$ as guaranteed by 
Theorem \ref{almost}(2) and pick $\vec y\in S^v$ 
such that $C\vec y=A\vec x$.
\end{proof}

\bibliographystyle{plain}

\end{document}